\documentclass{article}
\usepackage[english]{babel}
\usepackage[utf8]{inputenc}
\usepackage[T1]{fontenc}
\usepackage{graphicx}
\usepackage{geometry}
\usepackage{sectsty}
\usepackage{amsthm,amsfonts}
\usepackage{amsmath}
\usepackage{mathabx}
\usepackage{accents}
\usepackage{titlesec}
\usepackage{mathtools}
\mathtoolsset{showonlyrefs=true}

\titleformat*{\section}{\Large\bfseries}
\titleformat*{\subsection}{\large\bfseries}
\titleformat*{\subsubsection}{\large\bfseries}
\titleformat*{\paragraph}{\large\bfseries}
\titleformat*{\subparagraph}{\large\bfseries}

\newtheorem{teo}{Theorem}
\newtheorem{lema}[teo]{Lemma}

\newtheoremstyle{mytheoremstyle} 
{\topsep}                    
{\topsep}                    
{}                   
{}                           
{\scshape}                   
{.}                          
{.5em}                       
{}  

\theoremstyle{mytheoremstyle} \newtheorem{nota}{Remark}
\theoremstyle{mytheoremstyle} 
\numberwithin{equation}{section}
\newcommand{\real}{\mathbb{R}}
\newcommand{\complex}{\mathbb{C}}

\newcommand{\by}{\textbf{y}}

\newcommand \ben {\begin{equation}}
\newcommand \een {\end{equation}}
\newcommand \be {\begin{equation*}}
\newcommand \ee {\end{equation*}}
\newcommand \bi {\begin{itemize}}
\newcommand \ei {\end{itemize}}

\newenvironment{sproof}{%
  \proof}{\endproof}

\DeclarePairedDelimiter{\floor}{\lfloor}{\rfloor}
\DeclareMathOperator*{\partere}{Re}
\DeclareMathOperator*{\parteim}{Im}

\title{\textbf{Some $L^\infty$ solutions of the hyperbolic nonlinear Schrödinger equation and their stability}}
\author{Simão Correia and Mário Figueira}
\begin{document}
\maketitle
\begin{abstract}
Consider the hyperbolic nonlinear Schrödinger equation (HNLS) over $\real^d$
$$
iu_t + u_{xx} - \Delta_{\textbf{y}} u + \lambda |u|^\sigma u=0.
$$
We deduce the conservation laws associated with (HNLS) and observe the lack of information given by the conserved quantities. We build several classes of particular solutions, including \textit{spatial plane waves} and \textit{spatial standing waves}, which never lie in $H^1$. Motivated by this, we build suitable functional spaces that include both $H^1$ solutions and these particular classes, and prove local well-posedness on these spaces. Moreover, we prove a stability result for both spatial plane waves and spatial standing waves with respect to small $H^1$ perturbations.
\end{abstract}

\section{Introduction}

In this work, we shall consider the hyperbolic nonlinear Schrödinger equation
\begin{equation}\tag{HNLS}\label{HNLS}
iu_t + u_{xx} - \Delta_{\textbf{y}} u + \lambda |u|^\sigma u=0, 
\end{equation}
where 
$$\lambda\in\real, \quad 0<\sigma<4/(d-2)^+,\quad
u=u(t,x,\textbf{y}), (x,\textbf{y})\in \real^d,\ d\ge 2
$$
(we use the convention: $4/(d-2)^+=\infty$, if $d=2$; $4/(d-2)^+=4/(d-2)$, $d\ge 3$).
This is a special case of a larger class of dispersive equations, namely
\begin{equation}
iu_t - Lu + \lambda |u|^\sigma u=0,
\end{equation}
with $L$ a second-order differential operator on the spatial variables. In fact, \eqref{HNLS} corresponds to the case where the principal part of $L$ has a unique negative direction. It is important to notice that, when all directions are positive, $L$ is elliptic, which includes the well-known nonlinear Schrödinger equation
\begin{equation}\tag{NLS}\label{NLS}
iu_t - \Delta u + \lambda |u|^\sigma u=0.
\end{equation}

Physically, (HNLS) is related with deep water waves and plasma physics (for $d=2$) and nonlinear optics (for $d=3$); see \cite{sulem}, \cite{conti}. In the latter case, $t$ is to be interpreted as the propagation direction and $x$ is time. Due to its practical relevance, this equation has been analyzed numerically in several papers (see \cite{fedele}, \cite{sulem}). On the other hand, in a theoretical point of view, one has works on both the linear equation and on some special solutions to (HNLS) (see \cite{saut1}, \cite{saut2}, \cite{nahmod}, \cite{nanlu}, \cite{godet}). However, a qualitative theory is still to be discovered. In fact, even though there is a very complete set of techniques for \eqref{NLS}, the presence of a negative direction makes most of them unusable, as we shall see later on.

Here, we make some theoretical groundwork on the behaviour of solutions of the initial value problem associated with (HNLS),
\begin{equation}\label{pvi}
\left\{\begin{array}{ll}
iu_t + \qedsymbol u + \lambda |u|^\sigma u=0,& \qedsymbol:=\partial_{xx}^2 - \Delta_\by \\
u(0,x,\by)=u_0(x,\by) &.
\end{array}\right.
\end{equation}

There is one thing that the reader should keep in mind: most results and techniques presented here are also applicable for the (NLS). In view of this, there are two points in our work that we want to highlight: first, it is important to see which techniques are also available for the (HNLS), so that one has a good starting point for future research; second, the difficulty that arised from (HNLS) impeled us to find properties that have not been investigated in the (NLS) context, and so our results point out new strategies and ways of thinking about these equations. 

Now we present briefly the contents of this paper. In the first section, we study the invariances of (HNLS) and deduce the corresponding conservation laws. Contrary to the (NLS) setting, these conservation laws do not give a clear picture of the dynamics of the equation. We make use of this section to present a complete deduction of a "generalized pseudo-conformal" transformation in the case $\sigma=4/d$. This transformation turns out to relate in a very clear way the usual pseudo-conformal transform and the lens transform (\cite{carles}, \cite{tao}).

In section 3, we present a large family of $L^\infty$ (semiclassical) solutions, which present both global and blow-up behaviours (in the $L^\infty$ norm). To do this, one uses the generalized pseudo-conformal transformation applied to solutions of a special nonlinear wave equation.

In section 4, we turn our study to \textit{spatial plane wave} solutions, that is, solutions of the form $u(t,x,\by)=f(t,x-c\cdot\by)$. For these solutions, one may also observe global existence and blow-up. Furthermore, we study local well-posedness on spaces that include both spatial plane waves and $H^1$ solutions and study the stability of these plane waves regarding small $H^1$ perturbations. To our knowledge, these results are a novelty even in the (NLS) context.

In section 5, we make some considerations on hyperbolically radial solutions. These solutions will be defined on special regions in $\real^d$, over which one may observe several qualitative properties. The key ingredient is a reduction to the radial (NLS) equation, which allows to transfer results to this setting. The extension of such solutions to the whole space is a very difficult question (see remark \ref{naocola} and \cite{nahmod}).

Finally, in section 6, we consider \textit{spatial standing waves}, that is, solutions of the form $u(t,x,\by)=e^{i\omega x}\phi(t,\by)$. As in the spatial plane wave case, we study local well-posedness on spaces that include both spatial plane waves and $H^1$ solutions and study the $H^1$-stability of these spatial standing waves.

\begin{nota}
A motivation to consider spatial plane waves and spatial standing waves is the fact that, for some models in nonlinear optics (see \cite{sulem}), the negative direction $x$ actually represents \textit{time} and so these solutions are such that their \textit{time} evolution (in the physical model) is of a specific form. Actually, in this context, the term "spatial" can be a bit misleading, but we introduce it so that there is no confusion with the standard mathematical notion of plane waves and standing waves. These solutions do not present finite energy in the mathematical sense (see section 2); however, they do have finite energy on the plane $Oyz$, transverse to the axis of propagation, for each fixed physical time $x$. Note that such a question is not posed on the (NLS) model, since there is no dependence on the physical time. One could even argue that, since the energy of solutions of (NLS) has no dependence in time, the integration of the energy in the time variable is infinite. Therefore, in the context of nonlinear optics, the finite-energy assumption in both $x$ and $\by$ directions may be too restrictive and should be replaced by a finite-energy condition on the transverse plane. This is supported by the difficulty in finding examples of finite-energy solutions. Note that these solutions do exist; however, an explicit example is yet to be found.
\end{nota}

\section{Invariances and conservation laws}

One of the main tools in the development of a qualitative theory for a given equation is to figure out certain spatial quantities whose evolution in time has a very explicit form (in most cases, it turns out that they are, in fact, constant in time). As in (NLS), (HNLS) has a Lagrangian formulation, which enables one to obtain conservation laws from invariances of the equation (which are, usually, more obvious to determine).

Since (HNLS) and (NLS) are quite similar in an algebraic way, one might expect that invariances of (NLS) have a very close-related counterpart in the (HNLS) setting. In fact, one has the following invariances of the (HNLS):

\begin{enumerate}
\item Space-time translations: $v(x,\by,t)=u(x+x_0,\by+\by_0,t+t_0)$;
\item Gauge invariance; $v(x,\by,t)=e^{i\theta}u(x,\by,t)$;

\item Galilean invariance: $v(x,\by,t)=e^{\frac{i}{2}(ax-\mathbf{b}\cdot\by) - \frac{i}{4}(a^2-|\mathbf{b}|^2)t}u(x-at,\by-\mathbf{b}t,t)$;
\item Dilation invariance; $v(x,\by,t)=\lambda^{\frac{2}{\sigma}} u(\lambda x, \lambda \by, \lambda^2 t)$;

\item Hyperbolic invariance ($d=2$): $v(t,x,y)=u(t,x\cosh\alpha + y\sinh\alpha, x\sinh\alpha + y\cosh\alpha)$;
\end{enumerate}

\begin{nota}
Since one still has the dilation invariance, one may define (as in the (NLS) context) the notion of criticallity: given a Banach space $E$, one says that $\sigma$ is $E$-critical if the dilation invariance leaves the norm of $E$ invariant.
\end{nota}

A standard application of Noether's theorem implies the following conservation laws, which may be rigorously justified using the same process as in (NLS):

\begin{enumerate}
\item Conservation of energy and linear momentum:
$$
\frac{dE(u(t))}{dt}:=\frac{d}{dt}\left(\int \frac{|u_x(t)|^2}{2} - \frac{|\nabla_\by u(t)|^2}{2} - \frac{\lambda}{\sigma+2}|u(t)|^{\sigma+2}\right)=0,\quad \frac{d}{dt}\parteim\int \bar{u}(t)\nabla u(t) = 0;
$$
\item Conservation of mass:
$$
\frac{d}{dt}\int |u(t)|^2 = 0;
$$
\item Center of mass evolution law:
$$
\frac{d}{dt}\int x|u(t)|^2 = \parteim \int \bar{u}(t)u_x(t),\  \frac{d}{dt}\int y|u(t)|^2 = - \parteim \int \bar{u}(t)u_y(t)
$$
(notice that the conservation of linear momentum implies that the center of mass evolves linearly);
\item Virial identity (part I):
$$
\frac{d}{dt}\parteim \int \bar{u}(t)(xu_x(t) - \by\cdot \nabla_{\by} u(t)) = 4E(u_0) + \lambda\left(\frac{2d+4}{\sigma+2}-d\right)\int |u(t)|^{\sigma + 2}
$$
\end{enumerate}

\begin{nota}\label{energianaodefinida}
Note that the energy gives no direct information on the $L^2$-norm of the gradient. As a consequence, several techniques that are available for the (NLS) are unusable here. For example, the standard global existence result in the $L^2$-subcritical case relying on Gagliardo-Nirenberg's inequality is not applicable.
\end{nota}

\subsection{Deduction of the generalized pseudo-conformal invariance}\label{pseudo}

In the special case $\sigma=4/d$, one observes another family of invariances. These invariances will be of importance later on (see sections 3.3 and 4). Here, we present a complete deduction of such a family. This may also shed some light on the corresponding group of invariances for the (NLS).

Given smooth functions $u:\real\times \real^d\to \complex$ and $g,a,b,f:[0,T)\to\real$, define
\begin{equation}\label{lente}
v(t,x,\by)=u\left(g(t), \frac{x}{b(t)}, \frac{\by}{b(t)}\right)\exp\left(\frac{ia(t)(x^2-|\by|^2)}{4}\right)f(t)
\end{equation}

A tedious but straighfoward calculation shows that
\begin{align*}
& iv_t + \qedsymbol v + \lambda |v|^{4/d}v\\=& \Big( i\frac{f}{f'} u - i\frac{b'}{b^2}((x,\by)\cdot\nabla u) + iu_sg' - a'\frac{(x^2-|\by|^2)}{4}u + b^{-2}\qedsymbol u + iab^{-1}((x,\by)\cdot \nabla u)\\+&\frac{iad}{2}u - u\frac{a^2}{4}(x^2-|\by|^2) + \lambda|f|^{4/d}|u|^{4/d}u\Big)\times \exp\left(\frac{ia(x^2-|\by|^2)}{4}\right)f
\end{align*}

If we consider the restrictions
$$
(i)\ \frac{f'}{f}=-a\frac{d}{2},\quad (ii)\ b'=ab,\quad (iii)\ f(0)=b(0)=1
$$
(which imply that $f=b^{-d/2}$), we obtain
\begin{align*}
iv_t + \qedsymbol v + \lambda |v|^{4/d}v= \Big( &iu_sg' - \frac{a'+a^2}{4}(x^2-|\by|^2)u + b^{-2}\qedsymbol u \\&+ \lambda b^{-2}|u|^{4/d}u\Big)\times \exp\left(\frac{ia(x^2-|\by|^2)}{4}\right)f
\end{align*}

Finally, if we demand that, for some $k\in\real$,
$$
(iv)\ g'=b^{-2}, \quad (v)\ a'+a^2=4kb^{-4},
$$
one arrives to
\begin{align*}
& iv_t + \qedsymbol v + \lambda |v|^{4/d}v= \Big( iu_s + \qedsymbol u + \lambda|u|^{4/d}u- k(x^2-|\by|^2)u \Big)\times b^{-2}\exp\left(\frac{ia(x^2-|\by|^2)}{4}\right)f.
\end{align*}

Therefore $v$ is a solution of (HNLS) if and only if $u$ is a solution of the (HNLS) with an "harmonic" potential
$$
iu_s + \qedsymbol u + \lambda|u|^{4/d}u- k(x^2-|\by|^2)u = 0.
$$

Notice that, from $(ii)$ and $(v)$,
$$
a''+2a'a=-4kb^{-4}\left(4\frac{b'}{b}\right)= -(a'+a^2)4a.
$$
Hence
\begin{equation}\label{equacaoa}
a'' + 6aa' + 4a^3 =0
\end{equation}
with an initial condition $a(0)=a_0$ and, since $b(0)=1$, $a'(0)=4k-a_0^2$. It follows from $(i), (ii), (iii)$ and $(iv)$ that
$$
b(t)=\exp\left(\int_0^t a(\tau)d\tau\right),\ f(t)=\exp\left(-\frac{d}{2}\int_0^ta(\tau)d\tau\right),
$$
\begin{equation}\label{bfg}
g(t)=g(0)+\int_0^t\exp\left(-2\int_0^\rho a(\tau)d\tau\right)d\rho.
\end{equation}

Notice that the presence of $g(0)$ induces only a translation in time, which is an invariance we already covered, and so we shall consider $$(vi)\ g(0)=0.$$ In conclusion, the transform defined by \eqref{lente} with hypothesis $(i)-(vi)$ is a 2-parameter transform, the \textit{generalized pseudo-conformal transform}, which we denote by $v=\mathcal{T}_{a_0,k}u$. For example, choosing $a_0=0$ and $k\in\real^+$, we obtain from \eqref{equacaoa}
$$
a(t)=\frac{4kt}{4kt^2+1},\ b(t)=(4kt^2+1)^{1/2},\ g(t)=\frac{1}{\sqrt{4k}}\tan^{-1}(\sqrt{4k}t),\ f(t)=\frac{1}{(4kt^2+1)^{d/4}}.
$$
In the special case $k=1/4$, the transformation \eqref{lente} reads
$$
v(t,x)=\left(\mathcal{T}_{0,1/4}u\right)(t,x,\by)=\frac{1}{(1+t^2)^{d/4}}u\left(\tan^{-1}(t), \frac{x}{\sqrt{1+t^2}}, \frac{\by}{\sqrt{1+t^2}}\right)e^{i\frac{(x^2-|\by|^2)t}{4(1+t^2)}}
$$
which is the version for (HNLS) of the inverse of the lens transform (see \cite{carles}, \cite{tao}).

On the other hand, if $k=0$ and $a_0\in\real$, we get $a'+a^2=0$ and so
$$
a(t)=\frac{a_0}{1+a_0t},\ b(t)=1+a_0t,\ g(t)=\frac{t}{1+a_0t},\ f(t)=(1+a_0t)^{-d/2}.
$$
This means that the transformation $v=\mathcal{T}_{a_0,0}u$ is precisely the version for (HNLS) of the usual pseudo-conformal transform. This transform has a major relevance in the qualitative theory for (NLS), one of the reasons being that it gives rise to the Virial identity. In the context of (HNLS), the pseudo-conformal transform induces the conservation law
$$
\frac{d}{dt}\int (|x|^2-|\by|^2)|u(t)|^2=4\parteim \int \bar{u}(t)(xu_x - \by\cdot\nabla_{\by}u(t))
$$
which implies the Virial identity (or "variance" identity)
$$
\frac{d^2}{dt^2}V(u(t)):=\frac{d^2}{dt^2}\int (|x|^2-|\by|^2)|u(t)|^2=16E(u_0) + 4\lambda\left(\frac{2d+4}{\sigma+2}-d\right)\int |u(t)|^{\sigma + 2}.
$$

\begin{nota}
Similarly to remark \ref{energianaodefinida}, note that $V$ is not a definite functional, which overthrows, for example, the variance blow-up argument for the supercritical focusing equation. It becomes clear that there is a certain "competition" between the $x$-direction and the $\by$-direction (see \cite{sulem} for a discussion on this topic). However, it is not clear at all how does this competition evolve in time: such a problem remains open.
\end{nota}

\begin{nota}
As in the $L^2$-critical (NLS), the pseudo-conformal transform may be used to obtain global solutions of (HNLS) for small $H^2$ initial data for large $\sigma$. Specifically, if one considers $v=\mathcal{T}_{0,-1}u$, a straightfoward calculation shows that $u$ is a solution of \eqref{pvi} on $[0,T)$ if and only if $v$ is a solution of the corresponding nonautonomous equation
\begin{equation}\label{eqpseudo}
iv_t + v_{xx} - \Delta_\by v + \lambda (1-t)^{\frac{d\sigma-4}{2}}|v|^\sigma v=0
\end{equation}
on the interval $[0,\frac{T}{1+T})$. Therefore, the global existence for $u$ is equivalent to the existence of $v$ on $[0,1)$, which may be achieved using Strichartz estimates (see, for example, \cite{cazenaveweissler}).
\end{nota}

\section{A family of semiclassical solutions}

Here, we shall apply the transformation $\mathcal{T}_{a_0,k}$ (cf. section \ref{pseudo}) to obtain a family of solutions of (HNLS) in the critical case $\sigma = 4/d$. The idea will be to use bound-states of (HNLS) (which are never in $H^1$, as proved in \cite{saut2}). We say that a $u$ is a semiclassical solution of (HNLS) if $u\in C([0,T),L^\infty(\real^d))$ and $u$ satisfies (HNLS) in the distributional sense.
The following theorem displays a family of semiclassical solutions as well as their dynamics.

\begin{teo}
Let $A_0\in L^\infty(\real^d)$ be a solution of
\begin{equation}\label{boundstate}
\qedsymbol u + \lambda |u|^{4/d}u = (k(x^2-|\by|^2) + \gamma_0)u,\ k,\gamma_0\in\real.
\end{equation}
Then the (HNLS) admits the following family of semiclassical solutions:
\begin{align}\label{semiclassica1}
\psi(t,x,\by)=&\exp\left(-\frac{d}{2}\int_0^ta(\tau)d\tau\right)A_0\left((x,\by)\exp\left(-\int_0^t a(\tau)d\tau\right)\right)\nonumber\\\times &\exp\left(ia(t)\frac{x^2-|\by|^2}{4}\right)\exp\left(i\gamma_0\int_0^t\exp\left(-2\int_0^s a(\tau) d\tau\right)ds\right)
\end{align}
with $a'(t)+a(t)^2=4k\exp\left(-4\int_0^t a(\tau) d\tau\right)$, $a(0)=a_0\in\real$. Futhermore, if $A_0(0)\neq 0$ and
\begin{enumerate}
\item if $k<0$, the solution blows up in finite time in the $L^\infty$ norm, for any initial data $a_0$;
\item if $k=0$, the solution blows up at $T=-1/a_0$;
\item if $k>0$, the solution is global in time and its $L^\infty$ norm decays like $O(1/t)$.
\end{enumerate}
\end{teo}
\begin{proof}
First of all, notice that $\phi(t,x,\by)=e^{i\gamma_0 t}A_0(x,\by)$ is a solution of $$
iu_t + \qedsymbol u + \lambda|u|^{4/d}u- k(x^2-|\by|^2)u = 0.
$$
Hence, for any given $a_0\in\real$, 
$$\psi(t,x,\by)=(\mathcal{T}_{a_0,k}\phi)(t,x,\by)=e^{i\gamma_0 g(t)}\phi\left( \frac{x}{b(t)}, \frac{\by}{b(t)}\right)\exp\left(\frac{ia(t)(x^2-|\by|^2)}{4}\right)f(t)$$
is a solution of (HNLS). Using \eqref{bfg}, one arrives to the expression \eqref{semiclassica1}.

To study the long-time behaviour of these solutions, we first observe that, since $a''+6a'a+4a^3=0$,
$$
a(t)=\frac{t+c_2}{(t+c_2)^2 + c_1}, \mbox{ for some } c_1,c_2\in\real.
$$
It follows from a simple computation that $c_1=a'(0) + a(0)^2=4k$. Now one splits in two possibilities:
\begin{enumerate}
\item if $k\le 0$, then $a$ blows up at $t_0=\sqrt{4|k|}-c_2$ as $1/t^2$ for $k<0$ and as $1/t$ for $k=0$. Since $A_0(0)\neq 0$, it follows from \eqref{semiclassica1} that $\psi$ also blows up in the $L^\infty$ norm at time $t_0$;
\item if $k>0$, then $a$ is global in time and decays as $O(1/t)$, which implies in turn that $\psi$ is global in time and presents the same decay in $L^\infty$.
\end{enumerate}
\end{proof}
\begin{nota}
A different technique using an hydrodynamical approach was used by O. Rozanova (\cite{rozanova}) to obtain a similar class of solutions for the (NLS). In fact, we observe that such a technique results in the generalized pseudo-conformal transform $\mathcal{T}_{a_0,k}$, which endows this transformation with a concrete physical interpretation.
\end{nota}

\begin{nota}
In the particular case $k=0$, $\gamma_0<0$, $\lambda\ge 0$ and $d=2$, we obtain easily an $L^\infty$ solution of \eqref{boundstate}. This is an elementary consequence of the energy integral for the Klein-Gordon equation $u_{xx}-u_{yy} + \lambda |u|^2u = \gamma_0 u$,
\begin{equation}
E(u(x))=\int |u_x(x)|^2 dy + \int |u_y(x)|^2 dy + \lambda \int |u(x)|^4 dy - \gamma_0\int |u(x)|^2 dy=E(u(0)),
\end{equation}
which implies
$$
\|u\|_{L^\infty(\real^2)}\lesssim \sup_x \|u(x,\cdot)\|_{H^1(\real)}\le E(0).
$$
\end{nota}

\section{Spatial plane waves}

In this section, we shall consider spatial plane waves, that is, solutions of the form $u(t,x,\by)=f(t,x-c\cdot\by)$, where $c\in\real^{d-1}$, $c\neq 0$, is a fixed vector. For $u$ to be a solution to the (HNLS), one must have
\begin{equation}\label{plane}
if_t + (1-|c|^2)f_{zz} + \lambda|f|^\sigma f=0, \ f(0,z)=f_0(z)
\end{equation}

One of the interesting properties is that the size of $|c|$ determines the nature of the equation: fix, for example, $\lambda=1$. Then,
\begin{enumerate}
\item if $|c|<1$, \eqref{plane} is the focusing (NLS), which may exhibit blow-up phenomena;
\item if $|c|=1$, one may solve expliclty \eqref{plane}:
$$
f(t,z)=f(0,z)e^{i|f(0,z)|^\sigma t}, \ t\in\real
$$
We observe that such solutions verify $|f(t,z)|=|f(0,z)|, \forall t, z$, which means that these solutions are \textit{localized} (that is, their shape is not distorted by the flow of the equation);
\item  if $|c|>1$, \eqref{plane} is the defocusing (NLS), for which no blow-up solutions exist.
\end{enumerate}

\begin{nota}
Spatial plane waves also exist for the (NLS); however, the speed of the wave has no influence in the global dynamics. Furthermore, the existence of solutions with constant amplitude is a unique property of (HNLS).
\end{nota}

\begin{nota}\label{nplana}
If $d\ge 3$, fix $1\le n<d$. We write $\mathbf{x}=(x,\by_1,\by_2)\in \real\times\real^{d-n}\times \real^{n-1}$. One may also consider solutions of the form $u(t,\mathbf{x})=f(t,x-c\cdot \by_1,\by_2)$, $c\in\real^{d-n}$, which we shall call $n$-dimensional spatial plane waves. In this situation,
$$
if_t + (1-|c|^2)f_{zz} - \Delta_{\by_2} f + \lambda |f|^\sigma f=0.
$$
If $|c|<1$, then $f$ satisfies a (HNLS)-type equation; if $|c|\ge 1$, then one arrives to the (NLS) equation, where one may observe global existence and finite-time blow-up, depending on the sign of $\lambda$.
\end{nota}

\subsection{Local existence}\label{seccaoexistlocal}

We now develop a suitable functional framework that includes both solutions in $H^1(\real^d)$ and spatial plane waves. 
Given $c\in\real^{d-1}\setminus \{0\}$, define the space of spatial plane waves
\begin{equation}
X_c=\left\{u\in L^1_{loc}(\real^d): \exists f\in H^2(\real): u(x,\by)=f(x-c\cdot\by)\ a.e. \right\}
\end{equation}
and endow it with the norm $\|u\|_{X_c}=\|f\|_{H^2}$ (we say that $f$ is the profile of $u$). If $u_0\in X_c$, then the solution of
\begin{equation}\label{livre}
iu_t + \qedsymbol u = 0, u(0)=u_0
\end{equation}
is given by $u(t,x,\by)=(U(t)f_0)(x-c\cdot \by)$, where $f_0$ is the profile of $u_0$, and $U(t)$ is the Schrödinger group $e^{i(1-|c|^2)t\Delta}$ acting on $H^2(\real)$ for $|c|\neq 1$, and the identity map for $|c|=1$. It is clear that $u(t)\in X_c, \forall t$, and that
$$
\|u(t)\|_{X_c}=\|U(t)f_0\|_{H^2}=\|f_0\|_{H^2}=\|u_0\|_{X_c},
$$
which means that $U(t)$ induces naturally the group of isometries $e^{it\qedsymbol }$ on $X_c$. Set
$$
E=H^1(\real^d)\oplus X_c.
$$ 
Obviously, the group $e^{it\qedsymbol }$ is also defined on $E$, since it is simply the sum of the corresponding groups in each of the spaces. Moreover, define
$$
X_c'=\left\{u\in L^1_{loc}(\real^d): \exists f\in L^2(\real): u(x,\by)=f(x-c\cdot\by)\ a.e. \right\}
$$
and $E'=H^{-1}(\real^d)\oplus X'_c$.
\begin{teo}\label{existenciaplana}
Let $0<\sigma<4/(d-2)^+$. For every $u_0\in E$, there exists $T(u_0)>0$ and a unique solution of \eqref{pvi} $u\in C([0,T(u_0)), E)\cap C^1((0,T(u_0)),E')$ which depends continuously on $u_0$. Also, the blow-up condition holds in the sense that
$$
\lim_{t\to T(u_0)} \|u(t)\|_{E}=\infty, \ \mbox{ if } T(u_0)<\infty.
$$
\end{teo}
\begin{proof}

On $E$, one actually observes a decoupling of the (HNLS) equation: writing the solution $u$ as $v+\phi$, $v\in H^1(\real^d)$, $\phi\in X_c$, $\phi(x,\by)=f(x-c\cdot\by)$, one has
\begin{equation}\label{lados}
iv_t + \qedsymbol v + \lambda|v+\phi|^\sigma(v+\phi) - \lambda|\phi|^\sigma\phi=-(i\phi_t + \qedsymbol \phi + \lambda|\phi|^\sigma\phi).
\end{equation}

It is clear that the right hand side is in $X'_c$. We claim that the left hand side is in $H^{-1}(\real^d)$: the problem resides on the nonlinear part. Since $f\in H^2(\real)\hookrightarrow W^{1,\infty}(\real)$, one has $\phi\in W^{1,\infty}(\real^d)$. Therefore
$$
\left||v+\phi|^\sigma(v+\phi) - |\phi|^\sigma\phi\right|\lesssim |v|+|v|^{\sigma+1}
$$
Since $v\in H^1(\real^d)$, $|v|^{\sigma+1}\in L^{\frac{\sigma+2}{\sigma+1}}(\real^d)$, the nonlinear term may be written as a sum of a $L^2$ function with a $L^{\frac{\sigma+2}{\sigma+1}}$ function, both of them lying in $H^{-1}(\real^d)$, by Sobolev's injection.

Using the fact that $H^{-1}(\real^d)\cap X_c'=\{0\}$ (see Appendix A), one concludes that both sides of \eqref{lados} are zero:
\begin{align}\label{desacoplado1}
iv_t &+ \qedsymbol v + \lambda|v+\phi|^\sigma(v+\phi) - \lambda|\phi|^\sigma\phi=0\\\label{desacoplado2}
if_t &+ (1-c^2)f_{zz} + \lambda|f|^\sigma f=0
\end{align}

Fix an initial data $u_0=v_0+\phi_0$, where $\phi_0$ has profile $f_0$. For \eqref{desacoplado2}, one may use the $H^2$ local well-posedness result for the one dimensional (NLS) equation and obtain the profile $f$. Then, one focus on \eqref{desacoplado1} and solves for $v$. Since the nonlinear term in \eqref{desacoplado1} satisfies the estimates
\begin{align}
\left|(|v+\phi|^\sigma(v+\phi) - |\phi|^\sigma\phi)-(|w+\phi|^\sigma(w+\phi) - |\phi|^\sigma\phi)\right|&\lesssim (1+|v+\phi|^\sigma + |w+\phi|^\sigma)|v-w|\\&\lesssim (1+|v|^\sigma + |w|^\sigma)|v-w|\label{estimativaexistencia1}
\end{align}
and
\begin{align}
\left|\nabla(|v+\phi|^\sigma(v+\phi) - |\phi|^\sigma\phi)\right|&\lesssim (1+|v|^\sigma + |\phi|^\sigma)|\nabla v|,\label{estimativaexistencia2}
\end{align}
the local existence result for \eqref{desacoplado1} actually follows from the standard Kato's method applied in the (HNLS) context. 

%
%
%
\end{proof}

\begin{nota}
Recall that, in the case $|c|=1$, the explicit solution of \eqref{desacoplado2} is given by 
$$
f(t,z)=f(0,z)e^{i|f(0,z)|^\sigma t}, \ t\in\real,
$$
which implies that, if $f(0,\cdot)\in W^{1,\infty}(\real)$, then $f(t,\cdot)\in W^{1,\infty}(\real)$, for all $t>0$. Since no estimate of $e^{it\qedsymbol}$ over $X_c$ is necessary, one may actually weaken the definition of $X_c$ in this case, by replacing $f\in H^2(\real)$ with $f\in W^{1,\infty}(\real)$.
\end{nota}

\begin{nota}
Fix $d=2$ and $\sigma\ge 1$. We claim that one is able to extend the above local well-posedness result to $H^1(\real^2)\oplus X_{c_1}\oplus X_{c_2}$. In fact, the decoupling of the equation on each of the function spaces still holds (see Appendix): for $u=v+\phi+\psi$, $v\in H^1(\real^2)$, $\phi \in X_{c_1}$ with profile $f$ and $\psi\in X_{c_2}$ with profile $g$, one has
\begin{align}\label{des4}
iv_t &+ \qedsymbol v + \lambda\left(|v+\phi+\psi|^\sigma(v+\phi+\psi) - |\phi|^\sigma\phi - |\psi|^\sigma\psi\right)=0\\
if_t &+ (1-c_1^2)f_{zz} + \lambda|f|^\sigma f=0,\\\label{des6}
ig_t & + (1-c_2^2)g_{zz} + \lambda|g|^\sigma g=0.
\end{align}

We point out that the nonlinear term of the first equation is truly in $H^{-1}(\real^2)$: note that
\begin{align*}
&\left||v+\phi+\psi|^\sigma(\phi+\psi+v)-|\phi|^\sigma\phi - |\psi|^\sigma\psi\right|\\\le& \left||v+\phi+\psi|^\sigma(\phi+\psi+v)-|\phi+\psi|^\sigma(\phi+\psi)\right|+\left||\phi+\psi|^\sigma(\phi+\psi)-|\phi|^\sigma\phi - |\psi|^\sigma\psi\right|\\\lesssim &(|v|^\sigma + |\phi|^\sigma + |\psi|^\sigma)|v| + |\psi|^\sigma|\phi| + |\phi|^\sigma|\psi|.
\end{align*}
Using the fact that $\phi,\psi\in L^\infty(\real^2)$, the terms with $|v|$'s are treated as in the previous proof. It remains to check the last couple of terms. For example, take  $|\phi|^\sigma\psi$. Then
\begin{align*}
\int_{\real^2} |\phi(x,y)|^{2\sigma}|\psi(x,y)|^2dxdy&=\int_{\real^2} |f(x-c_1y)|^{2\sigma}|g(x-c_2y)|^2dxdy\\&= \frac{1}{|c_1-c_2|}\int_{\real^2} |f(w)|^{2\sigma}|g(z)|^2dwdz \le  \frac{1}{|c_1-c_2|}\|f\|_\infty^{2\sigma -2}\|f\|_2^2\|g\|_2^2
\end{align*}
and so $|\phi|^2\psi\in L^2(\real^2)\hookrightarrow H^{-1}(\real^2)$. In fact, \eqref{des4} expresses the interaction between spatial plane waves with different velocities, which turns out to be an $H^1$ function. 
\end{nota}

\begin{nota}
One may apply the construction made in this section to the case of $n$-dimensional spatial plane waves (cf. Remark \ref{nplana}), by replacing the profile space $H^2(\real)$ by $H^k(\real^n)$, with $2k-2>n$, so that $H^k(\real^n)\hookrightarrow W^{1,\infty}(\real^n)$.
\end{nota}

%
%
%
%
\subsection{Stability result}

The aim of this section is to prove that a large class of spatial plane waves is stable with respect to $H^1(\real^d)$ pertubations. The idea is to obtain a global existence result for small data in the large power case for equation \eqref{desacoplado1}. This is not trivial, as one may observe by analyzing \eqref{desacoplado1}: since there are both linear and quadratic in $v$, these lower order terms may disrupt the smallness of $v$. Moreover, it is not clear that the spatial plane wave substrate $\phi$ (which has both infinite mass and energy) does not increase indefinitely the $H^1(\real^d)$ component.

\begin{nota}
One may try to study the linear part of \eqref{desacoplado1}
\begin{equation}\label{linearpotencial}
iv_t + \qedsymbol v + \lambda |\phi|^2v + 2\phi\partere v\bar{\phi}=0,
\end{equation}
which is closely related to decay estimates of $L=i\qedsymbol + iV(t,x)[\cdot]$, where $V(t,x)[v]=\lambda |\phi|^2v + 2\phi\partere v\bar{\phi}$. If one drops the second term in the potential, one might be able to use the results of \cite{fujiwara1}, \cite{fujiwara2} to obtain such estimates. Their method is to consider approximations of the free group $e^{itL}$ given by Feynman's path integral and requires some physical interpretation on the effect of the potential on the motion of quantum particles. However, if one does not drop the second term, this physical interpretation becomes unclear. Moreover, a simple calculation shows that \eqref{linearpotencial} does not conserve the $L^2$ norm.
\end{nota}

Before we proceed, we recall the definiton of admissible pair: we say that the pair $(q,r)$ is admissible if
$$
\frac{2}{q}=d\left(\frac{1}{2}-\frac{1}{r}\right)
$$
with $2\le r\le 2d/(d-2)$ for $d\ge 3$ and $2\le r< \infty$ for $d=2$.

The following Strichartz estimates hold in the same way as for the (NLS) (see \cite{cazenave}):
\begin{enumerate}
\item For every admissible pair $(q,r)$,
$$
\|U(\cdot)\phi\|_{L^q(\real;L^r(\real^d))}\le C\|\phi\|_{2},\quad \forall \phi\in L^2(\real^d);
$$
\item Given two admissible pairs $(q,r)$ and $(\gamma,\rho)$ and an interval $I\subset \real$, let
$$
\Phi_f(t):=\int_{t_0}^t U(t-s)f(s)ds,\ t\in I, t_o\in\bar{I}.
$$
Then there exists $C>0$, independent on $I$, such that
$$
\|\Phi_f\|_{L^q(I;L^r(\real^d))}\le C\|f\|_{L^{\gamma'}(I;L^{\rho'}(\real^d))}, \quad \forall f\in L^{\gamma'}(I;L^{\rho'}(\real^d)).
$$
\end{enumerate}
\begin{teo}\label{teoestabilidade}
Let $\sigma=4$, $d=2$. Suppose that $c\in \real^{d-1}$ is such that $(|c|-1)\lambda>0$. Fix $\phi_0\in X_c$ and let $f_0$ be its profile. If $f_0\in H^2(\real)\cap L^2(\real, x^2dx)\cap W^{1,1}(\real)$, then the plane wave $\phi$ of (HNLS) with initial data $\phi_0$ is $H^1$-stable, i.e., 
$$
\forall \delta>0\ \exists\epsilon>0\ \|v_0\|_{H^1}<\epsilon \Rightarrow \| u-\phi\|_{L^\infty((0,\infty),H^1(\real^d))}<\delta,
$$
where $u$ is the (global) solution in $E$ of (HNLS) with initial data $v_0+\phi_0$.
\end{teo}

\begin{nota}
The result also holds for $\sigma\ge 4$ even: first, the linear term in $v$, which is of the form $|\phi|^\sigma v$, is well-behaved, since it has a stronger time decay; second, the terms with powers higher than $5$ in $v$ can be treated analogously to the quintic term.
\end{nota}

\begin{proof}

\textit{Step 1. Properties of $\phi$.} Recall that $\phi(t,x,y)=f(t,x-cy)$, where $f\in H^2(\real)$ is a solution of
$$
if_t + (1-c^2)f_{zz} + \lambda|f|^4f=0,\ f(0)=f_0.
$$
Since $(|c|-1)\lambda>0$, the above equation is in the defocusing case. Therefore, $f$ is global in $H^2(\real)$ (\cite[Theorem 5.3.1]{cazenave}). 

By hypothesis, $f_0\in H^1(\real)\cap L^2(\real, x^2dx)$ and so it follows from \cite[Theorem 7.3.1]{cazenave} that 
$$
\|f(t)\|_{L^\infty}\le \frac{C}{t^{1/2}}.
$$
Moreover,
$$\|f(t)\|_{L^\infty}\le C\|f(t)\|_{H^1}\le CE(f_0).
$$
Finally, we would like to estimate $\|\nabla f(t)\|_\infty$. From Duhamel's formula, one has
$$
\nabla f(t) = S(t)\nabla f_0 + \int_0^t S(t-s)\nabla(|f(s)|^4f(s))ds.
$$
Hence, for $t>2$,
\begin{align*}
\|\nabla f(t)\|_{L^\infty}\le & \|\nabla  f_0\|_{L^1} + \int_0^t \frac{1}{\sqrt{t-s}}\|f(s)\|_{L^\infty}^3\|f(s)\|_{L^2} \|\nabla f(s)\|_{L^2} ds\\
\lesssim & \| \nabla f_0\|_{L^1} + \int_0^1 \frac{1}{\sqrt{t-s}}E(f_0)^3\|f_0\|_{L^2} E(f_0) ds + \int_1^t \frac{1}{\sqrt{t-s}}\frac{1}{s^{3/2}}\|f_0\|_{L^2} E(f_0) ds\\
\lesssim & \| \nabla f_0\|_{L^1} + C(f_0)\left(\int_0^1 \frac{1}{\sqrt{t-s}} ds +\int_1^{t-1} \frac{1}{\sqrt{t-s}}\frac{1}{s^{3/2}} ds + \int_{t-1}^t \frac{1}{\sqrt{t-s}}\frac{1}{s^{3/2}} ds \right)\\
\lesssim & \|\nabla f_0\|_{L^1} + C(f_0)\left(\frac{1}{\sqrt{t-1}} + \int_1^{t-1}\frac{1}{s^{3/2}}ds + \frac{1}{(t-1)^{3/2}}\int_{t-1}^{t}\frac{1}{\sqrt{t-s}}\right)\\
\lesssim & \|\nabla f_0\|_{L^1} + C(f_0)\left(\frac{1}{\sqrt{t-1}} + 1 + \frac{1}{(t-1)^{3/2}}\right) \le C(f_0).
\end{align*}
For $t\le 2$, a similar procedure allows one to bound $ \|\nabla f(t)\|_{L^\infty}\le C(f_0)$. Therefore $\nabla f$ in uniformly bounded in $L^\infty$. Since $\phi(t,x,y)=f(t,x-cy)$, it follows that
\begin{equation}
\|\phi(t)\|_{L^\infty(\real^2)}\le \min\left\{\frac{C}{t^{1/2}}, CE(f_0)\right\},\quad \|\nabla \phi(t)\|_{L^\infty(\real^2)}\le C(f_0).
\end{equation}

\textit{Step 2. Setup.} Fix $v_0\in H^1(\real^2)$ and consider the corresponding solution $v$ of 
\begin{equation}\label{equacaov}
iv_t + \qedsymbol v + \lambda(|v+\phi|^4(v+\phi) -|\phi|^4\phi)=0.
\end{equation}
We recall that $v$ is defined on $(0,T(u_0))$, where $u_0=v_0+\phi_0$. Since $f$ is global in $H^2(\real)$, the blow-up alternative of theorem \ref{existenciaplana} then implies that, if $T(u_0)<\infty$,
$$
\|v(t)\|_{H^1}\to \infty, \ t\to T(u_0).
$$ 

For a given $\eta>0$, which shall be fixed later, take $T>0$ large enough so that
\begin{equation}
\|\phi(t)\|_{L^\infty}\le \eta,\ t\ge T.
\end{equation}
If $\|v_0\|_{H^1}<\epsilon$ is small enough, the fixed-point used in the local existence result implies that $v$ is defined on $(0,T+1)$ and that $\|v(t)\|_{H^1}\le 2\|v_0\|_{H^1}$, $0<t<T+1$. Therefore, one may, without loss of generality, take as initial data $u(T)=v(T)+\phi(T)$ and prove the stability result.

We develop the nonlinear part as
$$
\lambda(|v+\phi|^4(v+\phi) -|\phi|^4\phi)=\sum_{i=1}^5 g_i(v,\phi),
$$
where each $g_i$ has $i$ powers of $v$ and $5-i$ powers of $\phi$. Define, for $i=3,4,5$, $\rho_i=i+1$ and $\gamma_i$ such that $(\gamma_i,\rho_i)$ is an admissible pair. In particular, $\rho_3=\gamma_3=4$. Consider, for $0<t<T(u_0)$,
$$
h(t)=\|v\|_{L^\infty((0,t),H^1(\real^2))} + \sum_{i=3}^{5} \|v\|_{L^{\gamma_i}((0,t),W^{1,\rho_i}(\real^2))}.
$$
We write Duhamel's formula,
$$
v(t)=U(t)v_0 + \sum_{i=1}^{5}\int_0^t U(t-s)g_i(v(s),\phi(s))ds.
$$
Therefore, for any admissible pair $(q,r)$,
\begin{equation}\label{estimativa}
\|v\|_{L^q((0,t),W^{1,r}(\real^2))}\le C\|v_0\|_{H^1} + \sum_{i=1}^{5} \left\|\int_0^\cdot U(\cdot-s)g_i(v(s),\phi(s))ds\right\|_{L^q((0,t),W^{1,r}(\real^2))}.
\end{equation}
For the sake of simplicity, we shall omit both the temporal and spatial domains. In the next steps, we shall estimate each term of the sum by a suitable power of $h(t)$.

\textit{Step 3. Estimate of higher-order terms in $v$ on \eqref{estimativa}.} Here, we shall estimate
$$
\left\|\int_0^\cdot U(\cdot-s)g_i(v(s),\phi(s))ds\right\|_{L^q(W^{1,r})},\quad i=3,4,5.
$$
Take $i=3$. Then
\begin{align*}
\left\|\int_0^\cdot U(\cdot-s)g_3(v(s),\phi(s))ds\right\|_{L^q(W^{1,r})}&\lesssim \|g_3(v,\phi)\|_{L^{\gamma_3'}(W^{1,\rho_3'})} \\ &\lesssim \||v|^3|\phi|^2\|_{L^{4/3}(W^{1,4/3})}\lesssim \|v\|^3_{L^4(W^{1,4})} \\&\lesssim \|v\|_{L^{\gamma_3}(W^{1,\rho_3})}^3 \lesssim h(t)^3. 
\end{align*}

Now we treat the case $i=4,5$:
\begin{align*}
\left\|\int_0^\cdot U(\cdot-s)g_i(v(s),\phi(s))ds\right\|_{L^q(W^{1,r})}&\lesssim \|g_i(v,\phi)\|_{L^{\gamma_i'}(W^{1,\rho_i'})} \\&\lesssim \||v|^i|\phi|^{5-i}\|_{L^{\gamma_i'}(W^{1,\rho_i'})} \lesssim \|v\|_{L^{\mu_i}(L^{\rho_i})}^{i-1}\|v\|_{L^{\gamma_i}(W^{1,\rho_i})}
\end{align*}
where
$$
\mu_i=\frac{(i-1)(i+1)}{2}>\gamma_i.
$$
Then, through the interpolation $L^{\gamma_i} - L^{\mu_i} - L^\infty$ and the injection $H^1\hookrightarrow L^{\rho_i}$,
$$
\left\|\int_0^\cdot U(\cdot-s)g_i(v(s),\phi(s))ds\right\|_{L^q(W^{1,r})}\lesssim h(t)^i, \ i=4,5.
$$

\textit{Step 4. Estimate of the linear term in $v$.}
\begin{align*}
\left\|\int_0^\cdot U(\cdot-s)g_1(v(s),\phi(s))ds\right\|_{L^q(W^{1,r})}\lesssim \|g_1(v,\phi)\|_{L^1(H^1)}\lesssim \||v||\phi|^4\|_{L^1(H^1)}.
\end{align*}
Using the properties deduced in Step 2,
\begin{align*}
\||v||\phi|^4\|_{L^1(H^1)}&\lesssim \int_0^t \|\phi(s)\|_{L^\infty}^3\|\phi(s)\|_{W^{1,\infty}}\|v(s)\|_{H^1} ds\lesssim \|v\|_{L^\infty(H^1)}\left(\int_0^t \|\phi(s)\|_{L^\infty}^3 ds\right) \\&\lesssim
\|v\|_{L^\infty(H^1)}\|\phi\|_{L^\infty(L^\infty)}^{1/2}\left(1 + \int_1^t \frac{1}{s^{5/4}} ds\right) \lesssim \|v\|_{L^\infty(H^1)}\|\phi\|_{L^\infty(L^\infty)}^{1/2}\lesssim h(t)\eta^{1/2}
\end{align*}

\textit{Step 5. Estimate of the quadratic term in $v$.} 
Recalling that $\phi, \nabla \phi$ are bounded in $L^\infty(L^\infty)$, one has
\begin{align*}
& \left\|\int_0^\cdot U(\cdot-s)g_2(v(s),\phi(s))ds\right\|_{L^q(W^{1,r})}\lesssim \|g_2(v,\phi)\|_{L^{4/3}(W^{1,4/3})}\lesssim \||v|^2|\phi|^3\|_{L^{4/3}(W^{1,4/3})}\\\lesssim&
\left(\int_0^t \int |\phi|^4|v|^{8/3} + |\phi|^4|v|^{4/3}|\nabla v|^{4/3}+ |\phi|^{8/3}|v|^{8/3}|\nabla\phi|^{4/3} \right)^{3/4}\\\lesssim&
\left(\int_0^t \|\phi\|_{L^\infty}^{8/3}\int |v|^2 + |\nabla v|^2 + |v|^4 + |\nabla v|^4\right)^{3/4}\\\lesssim& \left(\int_0^t \|\phi\|_{L^\infty}^{8/3}\int |v|^2 + |\nabla v|^2 + \int_0^t \int |v|^4 + |\nabla v|^4\right)^{3/4}\\\lesssim& \left(\|v\|^2_{L^\infty(H^1)}\left(1+\int_1^t \frac{1}{s^{8/6}}ds\right) + \|v\|_{L^4(W^{1,4})}^4\right)^{3/4} \lesssim h(t)^{3/2} + h(t)^3.
\end{align*}

\textit{Step 6. Conclusion.} Putting together Steps 3, 4 and 5, there exists a universal constant $D$ such that
\begin{equation}\label{estimativah}
h(t)\le D\left(\|v_0\|_{H^1} + h(t)\eta^{1/4} + h(t)^{3/2} + h(t)^3 + h(t)^4 + h(t)^5\right).
\end{equation}
Now, for $\eta< 1/16D^4$,  we arrive at
\begin{equation}
h(t)\lesssim  \|v_0\|_{H^1} +  \left(h(t)^{3/2} + h(t)^3 + h(t)^4 + h(t)^5\right)
\end{equation}
If $\|v_0\|_{H^1}$ is sufficiently small, then the above inequality implies $h(t)\in [0,h_0]\cup [h_1, \infty)$, for some $h_0<\delta, h_1$. Since $h(0)=0$, by continuity, one has $h(t)<\delta$, for all $t<T(u_0)$. The blow-up alternative then implies that $T(u_0)=\infty$, which concludes the proof.
\end{proof}

Set $\sigma=2$ and consider the (NLS) equation in dimension two. Fix an initial data $f_0\in H^3(\real^2)\cap L^2(\real^2, (|x|^2 + |y|^2) dxdy)\cap W^{1,1}(\real^2)$. Since $\sigma=2$ is the $L^2(\real^2)$-critical exponent, \cite[Theorem 6.2.1]{cazenave} may be adapted to prove that, if $\|f_0\|_{H^3}$ is sufficiently small, then the corresponding solution $f$ is global and bounded in $H^3(\real^2)\hookrightarrow W^{1,\infty}(\real^2)$. Moreover, since $\Delta f$ is bounded in $L^2$, one may prove, using Duhamel's formula, that
$$
\|\nabla f(t)\|_\infty\lesssim \frac{1}{t^{1/2}}\|\nabla f_0\|_1,\ \mbox{ for } t \mbox{ large}.
$$
This decay allows one to prove the following:

\begin{teo}\label{teoestabilidade2}
Let $\sigma=2$, $d=3$. Fix a $2$-dimensional spatial plane wave $\phi$ of (HNLS), with speed $|c|>1$, and let $f_0$ be its initial profile. If $f_0\in H^3(\real^2)\cap L^2(\real^2, (|x|^2 + |y|^2) dxdy)\cap W^{1,1}(\real^2)$ has sufficiently small $H^3$ norm, then $\phi$ is $H^1$-stable, i.e., 
$$
\forall \delta>0\ \exists\epsilon>0\ \|v_0\|_{H^1}<\epsilon \Rightarrow \| u-\phi\|_{L^\infty((0,\infty),H^1(\real^d))}<\delta,
$$
where $u$ is the (global) solution in $E$ of (HNLS) with initial data $v_0+\phi_0$.
\end{teo}

\begin{nota}
One could try to prove a more general result for $n$-dimensional plane waves (cf. Remark \ref{nplana}) in $\real^d$ for $0<\sigma<4/(d-2)^+$. However, this turns out to be impossible using our technique:
\begin{enumerate}
\item First of all, the proof could only work for $\sigma$ even, since one needs to write the nonlinearity as a sum of powers of $|v|$ and apply a different Strichartz estimate to each one of them;
\item From the Sobolev critical exponent, if $d=3$, only $\sigma=2$ is an admissible even power; for $d\ge 4$, one must have $\sigma<2$, which excludes all even powers;
\item In the estimate of the linear term, one needs sufficient decay on $\|\phi\|_\infty$ so that $\|\phi\|^{\sigma}_\infty$ is integrable. However, for $n=1$ and $\sigma=2$, $\|\phi\|^{2}_\infty \approx 1/t$. 
\end{enumerate}
\end{nota}

\section{Hyperbolically symmetric solutions}

In this section, we focus our attention on a particular class of non-integrable solutions of \eqref{pvi} in $d=2$. These solutions are invariant for the hyperbolic invariance (cf. section 2), and so, for each time $t\in [0,T)$, they are constant on the hyperbolas $x^2-y^2=k$. We refer these solutions as having hyperbolic symmetry. More precisely, we look for solutions of \eqref{pvi} of the form
\begin{equation}
u(t,x,y)=\left\{\begin{array}{ll}
\Phi(t,\sqrt{x^2-y^2})=\Phi(t,r),& r^2=x^2-y^2\ge 0\\
\Psi(t,\sqrt{y^2-x^2})=\Psi(t,s), & s^2=y^2-x^2\ge 0
\end{array}\right.,
\end{equation}
with $\Phi,\Psi:[0,T)\times ]0,\infty[\to\complex$.

Fix now $\epsilon>0$. First, we restrict our analysis to the region
$$
D^1_\epsilon=\{(x,y)\in\real^2: x^2-y^2>\epsilon^2\}
$$
and consider the problem
\begin{equation}\label{probregiaoepsilon}
\left\{\begin{array}{ll}
iu_t^\epsilon + \qedsymbol u^\epsilon + \lambda|u^\epsilon|^\sigma u^\epsilon=0,& u^\epsilon=u^\epsilon(t,x,y), (x,y)\in D_\epsilon^1\\
u^\epsilon(0,x,y)=u^\epsilon_0(x,y), & (x,y)\in D^1_\epsilon\\
u^\epsilon(t,x,y)=0, & (x,y)\in \partial D_\epsilon^1,\ t\in [0,T)
\end{array}\right.
\end{equation}
In the following and for simplicity of notation, we shall drop the superscript on $u^\epsilon$. Next, we look for solutions of \eqref{probregiaoepsilon} with hyperbolic symmetry: $u(t,x,y)=\Phi(t,r)$, $r=\sqrt{x^2-y^2}>\epsilon$. Since
$$
\qedsymbol u=\Phi_{rr} + \frac{\Phi_r}{r}, \ r>\epsilon
$$
the (HNLS) becomes
$$
i\Phi_t + \Phi_{rr}+ \frac{\Phi_r}{r}  + \lambda|\Phi|^\sigma\Phi=0
$$
and if we set $\tilde{u}(t,x,y)=\Phi(t,\sqrt{x^2+y^2})$, it follows that problem \eqref{probregiaoepsilon} is formally equivalent to the radial (NLS) problem:
\begin{equation}\label{probNLSepsilon}
\left\{\begin{array}{ll}
i\tilde{u}_t + \Delta \tilde{u} + \lambda|\tilde{u}|^\sigma \tilde{u}=0,& \tilde{u}=\tilde{u}(t,x,y), (x,y)\in \Omega_\epsilon=\real^2\setminus \overline{B(0,\epsilon)}\\
\tilde{u}(0,x,y)=\tilde{u}_0(x,y), & (x,y)\in \Omega_\epsilon\\
\tilde{u}(t,x,y)=0, & (x,y)\in \partial \Omega_\epsilon,\ t\in [0,T)
\end{array}\right..
\end{equation}

It becomes clear that the solutions of \eqref{probregiaoepsilon} with hyperbolic symmetry are closely related to radial solutions of the (NLS) on the exterior domain $\Omega_\epsilon=\real^2\setminus \overline{B(0,\epsilon)}$. In a very similar way, setting
$$
D^2_\epsilon=\{(x,y)\in\real^2: y^2-x^2>\epsilon^2\},
$$
a solution with hyperbolic symmetry $v(t,x,y)=\Psi(t,s)$, $s=\sqrt{y^2-x^2}>\epsilon$, of
\begin{equation}\label{prob2regiaoepsilon}
\left\{\begin{array}{ll}
iv_t+ \qedsymbol v + \lambda|v|^\sigma v=0,& v=v(t,x,y), (x,y)\in D^2_\epsilon\\
v(0,x,y)=v_0(x,y), & (x,y)\in D^2_\epsilon\\
v(t,x,y)=0, & (x,y)\in \partial D^2_\epsilon,\ t\in [0,T)
\end{array}\right.
\end{equation}
corresponds, through the expression $\tilde{v}(t,x,y)=\Psi(t,\sqrt{x^2+y^2})$, to a solution of 
\begin{equation}\label{prob2NLSepsilon}
\left\{\begin{array}{ll}
i\tilde{v}_t - \Delta \tilde{v} + \lambda|\tilde{v}|^\sigma \tilde{v}=0,& \tilde{v}=\tilde{v}(t,x,y), (x,y)\in \Omega_\epsilon=\real^2\setminus \overline{B(0,\epsilon)}\\
\tilde{v}(0,x,y)=\tilde{v}_0(x,y), & (x,y)\in \Omega_\epsilon\\
\tilde{v}(t,x,y)=0, & (x,y)\in \partial \Omega_\epsilon,\ t\in [0,T)
\end{array}\right..
\end{equation}
\begin{nota}
Notice that the (NLS) in \eqref{probNLSepsilon} and \eqref{prob2NLSepsilon} concerns the focusing and defocusing cases, respectively (or vice-versa, according to the sign of $\lambda$). We shall take note of this when we consider the similar problem to \eqref{probregiaoepsilon} on the domain $D_\epsilon^1\cup D_\epsilon^2$.
\end{nota}
Finally, the problem on the domain $D^1_0=\{(x,y)\in \real^2: x^2-y^2>0\}$ (resp. $D^2_0=\{(x,y)\in \real^2: y^2-x^2>0\}$) will be considered as well, namely
\begin{equation}\label{probregiao0}
\left\{\begin{array}{ll}
iv_t+ \qedsymbol v + \lambda|v|^\sigma v=0,& v=v(t,x,y), (x,y)\in D^1_0 \ (\mbox{resp. }(x,y)\in D^2_0)\\
v(0,x,y)=v_0(x,y), & (x,y)\in D^1_\epsilon
\end{array}\right.
\end{equation}

Note that here the problem of finding solutions with hyperbolic symmetry ammounts to the study of the radial (NLS) in all of $\real^2$. Now we can state the following results:

\begin{teo}[Solutions with hyperbolic symmetry on $D^1_\epsilon$]\label{teodepsilon}
Let $\tilde{u}_0\in H^1_0(\Omega_\epsilon)$, $\Omega_\epsilon=\real^2\setminus \overline{B(0,\epsilon)}$ be a radial function, $\tilde{u}_0(x,y)=\Phi_0(\rho)$, with $\rho=\sqrt{x^2+y^2}$. Then, for $0<\sigma<\infty$, there exists $T(\Phi_0)>0$ and a semiclassical solution with hyperbolic symmetry of \eqref{probregiaoepsilon}
$$
u\in C([0,T(\Phi_0)); L^\infty(D^1_\epsilon)),
$$
with initial data $u(0,x,y)=\Phi_0(\sqrt{x^2-y^2})$, and one has the blow-up alternative:
$$
T(\Phi_0)<\infty\Rightarrow \limsup_{t\to T(\Phi_0)} \|u(t)\|_{\infty}=\infty.
$$
In addition, let us consider the following cases:
\begin{enumerate}
\item If $0<\sigma<4$, then $T(\Phi_0)=\infty$;
\item If $4\le \sigma <\infty$ and
\begin{enumerate}
\item $\lambda<0$, then $T(\Phi_0)=\infty$;
\item $\lambda>0$, let
$$
E(\Phi_0)=E(\tilde{u}_0):=\frac{1}{2}\int_{\Omega_\epsilon} |\nabla \tilde{u}_0|^2 - \frac{\lambda}{\sigma+2}\int_{\Omega_\epsilon} |\tilde{u}_0|^{\sigma+2}
$$
be the energy associated with \eqref{probNLSepsilon}. Assume that $\tilde{u}_0\in H^2(\Omega_\epsilon)\cap L^2(\Omega_\epsilon, (|x|^2+|y|^2)dxdy)$ and that the energy $E(\tilde{u}_0)$ verifies one of the following conditions:
\begin{enumerate}
\item $E(\tilde{u}_0)<0$;
\item $E(\tilde{u}_0)\ge 0$ and, for $\theta(x)=\frac{1}{2}|x|^2 - \epsilon^2\log|x|$,
$$
\parteim \int_{\Omega_\epsilon} (\nabla \theta\cdot \nabla \tilde{u}_0)\overline{\tilde{u}_0} >0,\ \left|\parteim \int_{\Omega_\epsilon} (\nabla \theta\cdot \nabla \tilde{u}_0)\overline{\tilde{u}_0}\right|^2\ge 8E(\tilde{u}_0)\int_{\Omega_\epsilon} \theta|\tilde{u}_0|^2;
$$
\end{enumerate}
Then $T(\Phi_0)<\infty$.
\end{enumerate}
\end{enumerate}
\end{teo}

\begin{teo}[Solutions with hyperbolic symmetry on the cone $D^1_0$]\label{teod0}
Let $\tilde{u}_0\in H^2(\real^2)$, be a radial function, $\tilde{u}_0(x,y)=\Phi_0(\rho)$, with $\rho=\sqrt{x^2+y^2}$. Then, for $0<\sigma<\infty$, there exists $T(\Phi_0)>0$ and a semiclassical solution with hyperbolic symmetry of \eqref{probregiaoepsilon}
$$
u\in C([0,T(\Phi_0)); L^\infty(D^1_0)),
$$
with initial data $u(0,x,y)=\Phi_0(\sqrt{x^2-y^2})$, and one has the blow-up alternative:
$$
T(\Phi_0)<\infty\Rightarrow \limsup_{t\to T(\Phi_0)} \|u(t)\|_{\infty}=\infty.
$$
In addition, let us consider the following cases:
\begin{enumerate}
\item If $0<\sigma<2$, then $T(\Phi_0)=\infty$;
\item If $2\le \sigma <\infty$ and
\begin{enumerate}
\item $\lambda\|\tilde{u}_0\|_2^2<4$, then $T(\Phi_0)=\infty$;
\item $\lambda>0$, let
$$
E(\Phi_0)=E(\tilde{u}_0):=\frac{1}{2}\int |\nabla \tilde{u}_0|^2 - \frac{\lambda}{\sigma+2}\int|\tilde{u}_0|^{\sigma+2}.
$$
Suppose that $\tilde{u}_0\in  L^2(\real^2, (|x|^2+|y|^2)dxdy)$ and that $E(\tilde{u}_0)$ verifies one of the following conditions:
\begin{enumerate}
\item $E(\tilde{u}_0)<0$;
\item $E(\tilde{u}_0)\ge 0$ and
$$
\parteim \int (x\cdot \nabla \tilde{u}_0)\overline{\tilde{u}_0} >0,\ \left|\parteim \int (x\cdot \nabla \tilde{u}_0)\overline{\tilde{u}_0}\right|^2\ge 4E(\tilde{u}_0)\int |x\tilde{u}_0|^2;
$$
\end{enumerate}
Then $T(\Phi_0)<\infty$ and, for $2\le\sigma<4$, the blow-up is taken at the cone $C=\{(x,y)\in\real^2: x^2-y^2=0\}$ in the sense that
\begin{equation}\label{concentracao}
\forall\epsilon>0\ \liminf_{t\to T(\Phi_0)} \|u(t)\|_{L^\infty(D^1_0\setminus D^1_\epsilon)}=\infty
\end{equation}

\end{enumerate}
\end{enumerate}
\end{teo}

\begin{proof}
We start with the proof of Theorem \ref{teodepsilon}. The local existence and uniqueness of solution for the problem \eqref{probNLSepsilon} with $\tilde{u}_0\in H^1_0(\Omega_\epsilon)$ is a consequence of \cite[Theorem 1]{tzvetkov}: there exists $T(\tilde{u}_0)=T(\Phi_0)>0$ and a unique maximal solution $\tilde{u}\in C([0,T(\Phi_0)); H^1_0(\Omega_\epsilon))$ of the problem \eqref{probNLSepsilon}. Since $\tilde{u}_0$ is radial symmetric, it follows by uniqueness that $\tilde{u}(t)$ is also radial symmetric for all $t\in [0,T(\Phi_0))$.

On the other hand, using the classical inequality for radial functions
\begin{equation}\label{desigualdaderadial}
\|\tilde{u}\|_{L^\infty(\Omega_\epsilon)} \le C\|\nabla \tilde{u}\|_{L^2(\Omega_\epsilon))}^{\frac{1}{2}}\|\tilde{u}\|_{L^2(\Omega_\epsilon))}^{\frac{1}{2}}.
\end{equation}
one has $\tilde{u}\in C([0,T_{max}), L^\infty(\Omega_\epsilon))$, $\tilde{u}=\Phi(t,\rho)$. Setting $u(t,x,y)=\Phi(t,\sqrt{x^2-y^2})$, it follows that $u$ is a solution of \eqref{probregiaoepsilon} in the distributional sense.

We recall that $\tilde{u}$ satisfies the following conservation laws:
\begin{equation}\label{conservacao}
\int_{\Omega_\epsilon}|\tilde{u}(t)|^2 = \int_{\Omega_\epsilon} |\tilde{u}_0|^2,\ E(\tilde{u}(t))=E(\tilde{u}_0),\ 0<t<T(\Phi_0).
\end{equation}

We derive, from \eqref{desigualdaderadial} and \eqref{conservacao},
\begin{align*}
\int_{\Omega_\epsilon} |\nabla \tilde{u}(t)|^2 &\le 2E(\tilde{u}_0) + \frac{2}{\sigma+2}\|\tilde{u}_0\|_{L^2(\Omega_\epsilon)}^2\|\tilde{u}(t)\|_{L^\infty(\Omega_\epsilon)}^\sigma\\&\le 2E(\tilde{u}_0) + \frac{2}{\sigma+2}\|\tilde{u}_0\|_{L^2(\Omega_\epsilon)}^{2+\frac{\sigma}{2}}\left(\int_{\Omega_\epsilon}|\nabla \tilde{u}(t)|^2\right)^{\frac{\sigma}{4}}.
\end{align*}
For $\sigma<4$, we obtain the control of the norm $\|\nabla \tilde{u}(t)\|_{L^2(\Omega_\epsilon)}$ and the solution is global.  The case 2.(a) is trivial by \eqref{conservacao}. Finally, under the assumptions of case 2.(b), it follows that the maximal time of existence, $T_{max}$, of the solution
$$
u\in C([0,T_{max}), H^2(\Omega_\epsilon)\cap H_0^1(\Omega_\epsilon))\cap C^1((0,T_{max}), L^2(\Omega_\epsilon)),
$$
is finite (cf. \cite[Proposition 1.6]{kavian}). We claim that this implies 
$$
\limsup_{t\to T_{max}} \|\tilde{u}(t)\|_{L^\infty(\Omega_\epsilon)}=\infty.
$$
Indeed, if this was not true, using Duhamel's formula and Gronwall's lemma, $\|\tilde{u}(t)\|_{H^2}$ should be bounded in $[0,T_{max})$, which is absurd. Hence $T(\Phi_0)=T_{max}$ and the proof of Theorem \ref{teodepsilon} is concluded.

The same procedure is used in Theorem \ref{teod0}, but now we have $\real^2$ instead of $\Omega_\epsilon$. Hence, the proof amounts to the well-known global existence and blow-up results for (NLS) in $\real^2$. In the case 2(b), $2\le\sigma<4$, we have the concentration of $\tilde{u}$ at the origin (cf. \cite[Remark 3.1]{merle}) in the sense that
$$
\forall \epsilon>0\ \liminf_{t\to T(\Phi_0)} \|u(t)\|_{L^\infty(\{x^2+y^2<\epsilon^2\})}=\infty,
$$
which implies \eqref{concentracao}.
\end{proof}
\begin{nota}
The above results are obviously valid for the domains $D^2_\epsilon$ and $D^2_0$: one must simply replace $\lambda$ by $-\lambda$.
\end{nota}

Finally, one may consider the problems
\begin{equation}\label{probregiaosemdiagonais}
\left\{\begin{array}{ll}
iu_t+ \qedsymbol u + \lambda|u|^\sigma u=0,& u=u(t,x,y), (x,y)\in D_\epsilon=D^1_\epsilon\cup D^2_\epsilon\\
u(0,x,y)=u_0(x,y), & (x,y)\in D_\epsilon\\
u(t,x,y)=0, & (x,y)\in \partial D_\epsilon, \ t\in [0,T)
\end{array}\right.
\end{equation}
and
\begin{equation}\label{probregiaodoisepsilon}
\left\{\begin{array}{ll}
iu_t+ \qedsymbol u + \lambda|u|^\sigma u=0,& u=u(t,x,y), (x,y)\in D_0=D^1_0\cup D^2_0 \\
u(0,x,y)=u_0(x,y), & (x,y)\in D_0
\end{array}\right.
\end{equation}
and build $L^\infty$ solutions by gluing solutions on regions $D^1_\epsilon$ and $D^2_\epsilon$ (resp. $D^1_0$ and $D^2_0$). Over $D_\epsilon$ (resp. $D_0$), in the case $\sigma<4$ (resp. $\sigma<2$), the solutions are always global. Otherwise, we note that the sign of $\lambda$ is not sufficient to guarantee global existence: if $\lambda>0$, then the equation is focusing on regions $D^1_0$ and $  D^1_\epsilon$; if $\lambda<0$, then it is focusing on $D^2_0$ and  $D^2_\epsilon$. In either case, one may observe finite-time blow-up.

\begin{nota}\label{naocola}
Fix a continuous initial data $u_0\in C(\real^2)$ with hyperbolic symmetry. Consider the radial (NLS) counterparts $\widetilde{u_0\big|_{D_0^1}}$ and $\widetilde{u_0\big|_{D_0^2}}$. Assuming these are $H^1$ functions, one may build a $L^\infty$ solution $u$ of (HNLS) on $D_0$ with initial data $u_0$. The question is wether the continuity of $u_0$ over $\{|y|=|x|\}$ remains valid for $u$ (in the sense that $u$ admits a continuous extension to $\real^2$). The answer is, in general, negative: take $Q$ to be the positive radial ground-state of (NLS) in $\real^2$ and consider the pseudo-conformal transform of $V=e^{it}Q$,
$$
W(t,x,y)= (1-t)^{-1}V\left(\frac{t}{1-t}, \frac{x}{1-t}, \frac{y}{1-t}\right)\exp\left(-i\frac{(x^2+y^2)}{4(1-t)}\right).
$$
One easily checks that $W$ is radial, $W(0,0,0)=V(0,0,0)=Q(0)$ and $|W(t,0,0)|=(1-t)^{-1}|V(t,0,0)|=(1-t)^{-1}Q(0,0)$. Then, setting
$$
u_0(x,y)=\left\{\begin{array}{ll}
Q(\sqrt{x^2-y^2}), & x^2-y^2>0\\
e^{-i\frac{y^2-x^2}{4}}Q(\sqrt{y^2-x^2}), & y^2-x^2>0
\end{array}\right.,
$$
the corresponding solution is given by
$$
u(t,x,y)=\left\{\begin{array}{ll}
V(t,\sqrt{x^2-y^2}), & x^2-y^2>0\\
W(t,\sqrt{y^2-x^2}), & y^2-x^2>0
\end{array}\right.,
$$
which is not continuous at $\{|y|=|x|\}$ for any positive time $t>0$.
\end{nota}

\section{Spatial standing waves}

One of the ways to overcome the presence of a negative direction is to search for solutions of the (HNLS) of the form $u(t,x,\by)=e^{i\omega x}\phi(t,\by)$, (somehow in analogy to the usual notion of bound-state - recall that in some models of nonlinear optics, these are truly time-periodic solutions). Inserting this expression into the equation,
\begin{equation}
i\phi_t -\omega^2 \phi - \Delta_\by \phi + \lambda |\phi|^{\sigma}\phi =0.
\end{equation}
Setting $v(t,\by)=e^{-i\omega^2 t}\phi(-t,\by)$, one arrives to
\begin{equation}
iv_t  + \Delta_\by v - \lambda |v|^{\sigma}v =0.
\end{equation}
which is the (NLS) in $\real^{d-1}$. Consider the initial value problem
\begin{equation}
iv_t  + \Delta_\by v - \lambda |v|^{\sigma}v =0, v(0,\by)=v_0(\by)\in H^1(\real^{d-1}).
\end{equation}
As it is well-known,
\begin{enumerate}
\item for $\lambda>0$ or $\sigma<4/(d-1)$, one has global existence of solutions in $H^1(\real^{d-1})$;
\item for $\lambda<0$ and $\sigma>4/(d-1)$, initial data $v_0\in H^1(\real^{d-1})\cap L^2(\real^{d-1}, |\by|^2d\by)$ with negative energy blows up in finite time.
\end{enumerate}
\subsection{Local existence and stability}
As in section \ref{seccaoexistlocal}, one may build a local well-posedness theory to include both $H^1$ solutions and spatial standing waves. Set $k=\floor{\frac{d+1}{2}}+1$.
If one defines
\begin{equation}
Y_\omega=\left\{\phi\in L^1_{loc}(\real^d): \exists f\in H^k(\real^{d-1}): \phi(x,\by)=e^{i\omega x}f(\by)\ a.e.\right\},
\end{equation}
\begin{equation}
Y'_\omega=\left\{\phi\in L^1_{loc}(\real^d): \exists f\in H^{k-2}(\real^{d-1}): \phi(x,\by)=e^{i\omega x}f(\by)\ a.e.\right\}
\end{equation}
and set $F=H^1(\real^d)\oplus Y_\omega$ and $F'=H^{-1}(\real^d)\oplus Y'_\omega$, then one has the following
\begin{teo}\label{existenciastanding}
Let $0<\sigma<4/(d-2)^+$. For every $u_0\in F$, there exists $T(u_0)>0$ and a unique solution of \eqref{pvi} $u\in C([0,T(u_0)), F)\cap C^1((0,T(u_0)),F')$ which depends continuously on $u_0$. Also, the blow-up condition holds in the sense that
$$
\lim_{t\to T(u_0)} \|u(t)\|_{F}=\infty,\ \mbox{ if } T(u_0)<\infty.
$$
\end{teo}
\begin{sproof}
The proof is almost identical to that of Theorem \ref{existenciaplana}. First of all, since $H^{-1}(\real^d)\cap Y'_\omega=\emptyset$, a solution $u=v+\phi\in F$ of (HNLS) with initial data $u_0=v_0+\phi_0\in F$, $\phi_0(\by)=e^{i\omega x}f_0(\by)$, is equivalent to a solution of the system
\begin{align}\label{desacoplado3}
iv_t &+ \qedsymbol v + \lambda|v+\phi|^\sigma(v+\phi) - \lambda|\phi|^\sigma\phi=0,\quad v(0)=v_0\\\label{desacoplado4}
if_t &+ \Delta_{\by}f - \lambda|f|^\sigma f=0, \quad f(0)=f_0, \quad \phi(t,x,\by)=e^{i\omega x-i\omega^2 t}f(-t,\by).
\end{align}
One then proceeds to solve \eqref{desacoplado4} \textit{backwards} in time using the usual $H^2$ local well-posedness results for (NLS). Finally, the fact that $f\in W^{1,\infty}(\real^2)$ and the estimates \eqref{estimativaexistencia1} and \eqref{estimativaexistencia2} allow the use of Kato's method to build the unique solution $v$ of \eqref{desacoplado3}.
\end{sproof}
Moreover, one may also derive $H^1$-stability for spatial standing waves, in a completely analogous fashion:
\begin{teo}
Fix $\omega\in\real$, $\lambda>0$, and set $\sigma= 4$, for $d=2$, $\sigma=2$ for $d=3$. Given $\phi_0\in Y_\omega$, suppose that its profile $f_0$ satisfies $f_0\in L^2(|\by|^2d\by)\cap H^2(\real^{d-1}) \cap W^{1,1}(\real^{d-1})$. Let $\phi$ be the spatial standing wave with initial data $\phi_0$. If either
\begin{enumerate}
\item $d=2$;
\item $d=3$, $f_0\in H^3(\real^2)$ and $\|f_0\|_{H^3}$ small;
\end{enumerate}
then $\phi$ is $H^1$-stable, i.e.,
$$
\forall \delta>0\ \exists\epsilon>0\ \|v_0\|_{H^1}<\epsilon \Rightarrow \| u-\phi\|_{L^\infty((0,\infty),H^1(\real^d))}<\delta,
$$
where $u$ is the (global) solution of (HNLS) in $F$ with initial data $v_0+\phi_0$.
\end{teo}
\begin{sproof}
Setting $\phi(t,x,\by)=e^{i\omega x-i\omega^2 t }f(-t,\by)$, one sees that 
$$
\|\phi(t)\|_{W^{1,\infty}(\real^d)}\lesssim \|f(-t)\|_{W^{1,\infty}(\real^{d-1})}
$$
and that $f$ satisfies a defocusing (NLS) in dimension $d-1$. As in the spatial plane wave case, this implies that $\phi$ is global and that
$$
\|\phi(t)\|_{L^\infty}\lesssim \frac{1}{t^{d/2}},\quad \|\nabla \phi\|_{L^\infty}\lesssim \frac{1}{t^{(d-1)/2}}.
$$
The proof then follows from Steps 2-6 in the proof of Theorem \ref{teoestabilidade}, with precisely the same estimates.
\end{sproof}

\begin{nota}
Spatial standing waves are solutions which lie on $H^1(\mathbb{T}\times \real^{d-1})$, and so one could simply try to extend the (NLS) results over this space (see, for example, \cite{takaoka}). We find our approach more interesting for its novelty and because it allows to understand the effect of $H^1$ perturbations on these special solutions.
\end{nota}

\section{Further comments}

Let us summarize some interesting questions that rise from this work:
\begin{enumerate}
\item The $H^1$ framework may not be suited for some physical models. The question then is: what is a suited framework? The spaces $E$ and $F$, built from some classes of solutions, indicate that a local well-posedness theory may be presented in such a way that it includes functions without decay at infinity. Even in a mathematical perspective, the "bound-state" solutions built in \cite{nahmod} and \cite{nanlu} do not lie in $H^1$. It would be interesting to find local-wellposednes on more general spaces, which do not demand decay at infinity.

\item A blow-up solution in the $H^1$ framework is yet to be found. Our solutions never possess sufficient decay to assure integrability. An example of blow-up would be of extreme importance.

\item The construction of spaces $E$ and $F$ has a great capacity of generalization: let $\mathcal{Z}$ be the class of functions which have a particular shape. For a given equation, which is locally well-posed in $\mathcal{X}$, suppose that one has a class of solutions in $\mathcal{Z}$. This will imply that the profile of these solutions verifies a reduced equation, for which one may have local existence over some space $\mathcal{Y}$. If $\mathcal{Z}\cap\mathcal{X}=\emptyset$, then one should be able to prove local well-posedness on
$$
\mathcal{E}=\mathcal{X}\oplus \{u\in\mathcal{Z}: \mbox{ the profile of } u \mbox{ is in } \mathcal{Y}\}.
$$
Furthermore, this allows one to obtain a suitable functional framework to study the effect of $\mathcal{X}$-perturbations on solutions in $\mathcal{Z}$.
\end{enumerate}

\section{Acknowledgements}
Mário Figueira was partially supported by Fundação para a Ciência e Tecnologia, through the grant UID/MAT/04561/2013. Simão Correia was also supported by Fundação para a Ciência e Tecnologia, through the grants SFRH/BD/96399/2013 
and UID/MAT/04561/2013. The authors are indebted to Rémi Carles for having called our attetion to this problem.
\appendix
\section{Appendix}

Recall the definition of $X'_c$:
\begin{equation}
X'_c=\left\{u\in L^1_{loc}(\real^d): \exists f\in L^2(\real): u(x,\by)=f(x-c\cdot\by)\ a.e. \right\}.
\end{equation}
We note that this is a correct definition: if $f,\tilde{f}:\real\to\complex$ differ in a zero-measure set, then
$$
g(x,\by)=f(x-c\cdot \by),\quad \tilde{g}(x,\by)=\tilde{f}(x-c\cdot \by)
$$
differ also in a zero-measure set (for the Lebesgue measure in $\real^d$).
For a fixed $h\in\real^{d-1}$, define the translation operator $T_h:L^1_{loc}(\real^d)\to L^1_{loc}(\real^d)$,
\begin{equation}
(T^c_hu)(x,\by):=u(x+c\cdot h,\by+h) \ a.e.\ (x,\by)\in \real^d.
\end{equation}
If $\phi\in X_c'$, then
$$
\phi(x,\by)=f(x-c\cdot \by)=f((x+c\cdot h) - c\cdot(\by + h))= (T_h\phi)(x,\by), \ a.e.\ (x,\by)\in \real^d.
$$
and therefore $T^c_h\phi=\phi$, $\forall \phi\in X_c'$.
\begin{lema}\label{lemaapendice}
One has $H^{-1}(\real^d)\cap X_c'=\{0\}$.
\end{lema}
\begin{proof}
Take $w\in H^{-1}(\real^d)\cap X_c'$. Then there exist $v_i\in L^2(\real^d)$, $i=0,...,d$, such that
$$
w=v_0 + \sum_{i=1}^{d} (v_i)_{x_i}
$$
and, for any $\Omega\subset \real^d$ open,
$$
\|w\|_{H^{-1}(\Omega)}=\left(\sum_{i=0}^{d}\int_\Omega |v_i|^2\right)^{1/2}.
$$
Take $\Omega=\{(x,\by)\in\real^d: |x|<1\}$. Fix $h\in \real^{d-1}$ such that $c\cdot h=1$. Using the fact that $T^c_hw=w$, one has
\begin{equation}
\|w\|_{H^{-1}(\real^d)}^2=\sum_{i=0}^{d}\int_{\real^d} |v_i|^2=\sum_{i=0}^{d}\sum_{m\in\mathbb{Z}} \int_{\Omega} |T^c_{2mh}v_i|^2=\sum_{m\in\mathbb{Z}} \|T^c_{2mh}w\|_{H^{-1}(\Omega)}^2=\sum_{m\in\mathbb{Z}} \|w\|_{H^{-1}(\Omega)}^2.
\end{equation}
Therefore one must have $\|w\|_{H^{-1}(\Omega)}=0$ and so $\|w\|_{H^{-1}(\real^d)}=0$.
\end{proof}
\begin{lema}
For any $c_1,c_2\in \real^{d-1}\setminus\{0\}$, $(H^{-1}(\real^d)\oplus X_{c_1})\cap X_{c_2}=\{0\}$.
\end{lema}
\begin{proof}
Take $z\in (H^{-1}(\real^d)\oplus X_{c_1})\cap X_{c_2}$. We write $z=w+\phi_1$, with $w\in H^{-1}(\real^d)$ and $ \phi_1\in X_{c_1}$. Fix $h\in\real^{d-1}$. Then
$$
T_h^{c_2}z=z,\ i.e.,\ w-T_h^{c_2}w=-\phi_1 + T_h^{c_2}\phi_1.
$$
The r.h.s. is in $H^{-1}(\real^d)$, while the l.h.s. is in $X_{c_1}$. Therefore both sides are equal to 0:
$$
w=T_h^{c_2}w,\ \phi_1=T_h^{c_2}\phi_1.
$$
One now concludes that $w,\phi_1=0$ as in the previous proof. 
\end{proof}

\small
\noindent \textsc{Sim\~ao Correia}\\
CMAF-CIO and FCUL \\
\noindent Campo Grande, Edif\'icio C6, Piso 2, 1749-016 Lisboa (Portugal)\\
\verb"sfcorreia@fc.ul.pt"\\

\small
\noindent \textsc{Mário Figueira}\\
CMAF-CIO and FCUL \\
\noindent Campo Grande, Edif\'icio C6, Piso 2, 1749-016 Lisboa (Portugal)\\
\verb"msfigueira@fc.ul.pt"\\

\end{document}